\documentclass[10pt]{asme2ej}

\usepackage{amsmath,amssymb}

\usepackage{amsmath,amssymb,amsthm}

\newtheorem{theorem}{Theorem}
\newtheorem{lemma}[theorem]{Lemma}

\newtheorem{definition}{Definition}
\newtheorem{remark}{Remark}
\newtheorem{example}{Example}

\newcommand{\ds}{\displaystyle}
\DeclareMathOperator{\Ima}{Im}

\title{Enlarged Controllability of Riemann--Liouville Fractional Differential Equations\thanks{This 
is a preprint of a paper whose final and definite form is with 
\emph{Journal of Computational and Nonlinear Dynamics}, 
ISSN 1555-1415, eISSN 1555-1423, CODEN: JCNDDM. 
Submitted 10-Aug-2017; Revised 28-Sept-2017 and 24-Oct-2017; Accepted 05-Nov-2017.}}

\author{Touria Karite\\
\affiliation{PhD Student\\
TSI Team, MACS Laboratory\\ 
Department of Mathematics\\ 
and Computer Science\\ 
Institute of Sciences\\ 
Moulay Ismail University\\ 
Meknes, Morocco\\
Email: touria.karite@gmail.com}}

\author{Ali Boutoulout\\
\affiliation{
Full Professor\\
Coordinator of TSI Team \& MACS Laboratory}\\ 
Department of Mathematics\\ 
and Computer Science\\ 
Institute of Sciences\\ 
Moulay Ismail University\\ 
Meknes, Morocco\\
Email: boutouloutali@yahoo.fr}

\author{Delfim F. M. Torres\thanks{Corresponding author.}\\
\affiliation{ 
Full Professor, Coordinator of the Systems and Control Group\\
Center for Research \& Development in Mathematics and Applications (CIDMA)\\
Department of Mathematics, University of Aveiro, 3810-193 Aveiro, Portugal\\
Email: delfim@ua.pt}}

\begin{document}
	
\maketitle    


\begin{abstract}
{\it We investigate exact enlarged controllability
for time fractional diffusion systems of Riemann--Liouville type. 
The Hilbert uniqueness method is used to prove exact enlarged 
controllability for both cases of zone and
pointwise actuators. A penalization method is given 
and the minimum energy control is characterized.}
\end{abstract}

{\it {\bf Keywords:} Regional controllability; Time fractional diffusion processes; 
Fractional calculus; Enlarged controllability;\\ 
\indent Minimum energy; Optimal control.}



\section{Introduction}

The purpose of fractional calculus is to generalize 
standard derivatives into non-integer order operators. 
As well acknowledged in the literature, many dynamical 
systems are best characterized by dynamic models of fractional order, 
based on the notion of non-integer order differentiation or integration. 
The study of fractional order systems is, however, more delicate. 
Indeed, fractional systems are, on one hand, memory systems, 
notably for taking into account the initial conditions, 
and on the other hand they present much more complex dynamics
\cite{Baleanu:book,M:O:T:book,SahaRay:book}.

Since traditional computation is based on differentiation 
and integer order integration, the concept of fractional computation 
has an enormous potential to change the way we see, model, and control 
the ``nature'' around us \cite{Baleanu:book,book:Almeida:Po:To,Gujarathi:2017}. 
Several theoretical and experimental studies 
show that some electrochemical \cite{RD-JN-1997,KBO-2010}, 
thermal \cite{JB-LLL-JCB-AO-OC-2000} and viscoelastic \cite{RLB-1983,GS-2007} systems 
are governed by non-integer order differential equations. The use of classical models, 
based on a complete differentiation, is therefore not appropriate. As a result, 
models based on non-integer differential equations have been strongly developed 
in the past few decades \cite{Bal:Wu:Zeng:2017,OC-AO-EB-JLB-2000,book:Mal:To}.

To backtrack the root of fractional calculus, one needs to go back to the 
XVII century (more accurately to 1695), when L'H\^{o}pital was questioning 
Leibniz about the possible meaning of half order differentiation. 
This question has attracted the interest of many well-known mathematicians, 
including Euler, Liouville, Laplace, Riemann, Gr\"{u}nwald, Letnikov, 
and many others. Since the XIX century, the theory of fractional 
calculus developed rapidly, mostly as a foundation for a number of mathematical disciplines, 
including fractional geometry, fractional differential equations (FDE), and fractional dynamics. 
Since 1974, when the first international conference in the field took place,
fractional calculus has been intensively developed with respect to practical applications.
Applications of fractional calculus are nowadays very wide and most disciplines 
of modern engineering and science rely on tools and techniques of fractional calculus. 
For example, fruitful applications can be found in rheology, viscoelasticity, acoustics, 
optics, chemical and statistical physics, robotics, control theory, electrical 
and mechanical engineering, and bioengineering \cite{RH-2000,JC-2014,RLM-2006}.
In \cite{Rev2:01}, a fractional derivative sub-diffusion equation model 
is proposed to describe prime number distribution. Based on experimental 
results, \cite{Rev2:02} develops a fractional damping wave equation
for the acoustic propagation in a porous media. For a survey of time 
fractional diffusion models under different initial and boundary conditions 
we refer the reader to \cite{Rev2:03}.

The problem of steering a system to a target state
has a vast literature, which has essentially began with the introduction 
of the notion of actuators and sensors in the 1980s \cite{AE-AJP-1988}. 
However, in many real world applications, one is concerned with those 
cases where the target states of the studied problem 
are defined in a given subregion of the whole space domain. 
This leads to the idea of regional controllability \cite{Th:EZ:1993,AEJ:AJP:MCS:EZ:1995}.
		
Controllability problems for integer order systems have been widely 
studied and many techniques have been developed for solving such problems 
\cite{book:Lions:1971,Berg:1992}. Problems with hard constraints 
on the state or control have attracted several authors in the last three decades, 
mostly for their importance in various applications in optimal control. 
Indeed, it is well known that by proving the existence of a Lagrange multiplier 
associated with the constraints in the state, we can derive 
optimality conditions \cite{Bon:Cas:1984:85,book:Kazu:Kuni:2008}. 
For instance, Barbu and Precupanu \cite{book:Bar:Precup:2012} 
and Laseicka \cite{Laseicka:1980}, derived the existence of 
a Lagrange multiplier for some optimal control problems
with integral state constraints. Bergounioux \cite{Berg:1992} 
used a penalization method to prove existence of a multiplier 
and to derive optimality conditions for elliptic equations 
with state constraints. Here, we solve the problem of enlarged 
controllability, also so-called controllability with constraints on the state, 
using the Hilbert Uniqueness Method (HUM) 
of Lions \cite{book:Lions:1971,JLL-1988}.	

We consider a time fractional diffusion system 
where the traditional first-order time derivative is replaced 
by the Riemann--Liouville time fractional derivative, 
which can be used to well characterize sub-diffusion processes.
The problem is defined in Section~\ref{sec:2}.
Some basic knowledge of fractional calculus and some preliminary 
results for our problem are given in Section~\ref{sec:3}. Then,
in Section~\ref{sec:4}, we characterize the enlarged controllability 
of the system. Our main results on enlarged controllability are proved 
in Section~\ref{sec:5}, in two different cases: for zone and pointwise actuators.
In Section~\ref{sec:6}, we prove that the control that steers the system 
to the final state is of minimum energy. We end with Section~\ref{sec:7} 
of conclusions and some interesting open questions 
that deserve further investigations.


\section{The Time Fractional Diffusion System}
\label{sec:2}

Let $\Omega$ be an open bounded subset of $\mathbb{R}^{n}$, $n=1,2,3$, 
with a regular boundary $\partial{\Omega}$. For $T>0$, denote 
$Q=\Omega\times[0,T]$ and $\Sigma=\partial\Omega\times[0,T]$. 
We consider the following time fractional order diffusion system:
\begin{equation}
\label{sys1}
\begin{cases}
\ds {}_{0}D^{\alpha}_{t}y(x,t) 
= \mathcal{A}y(x,t) + \mathcal{B}u(t) \quad &\mbox{in}\quad Q,\\
y(\xi,t)=0  \quad &\mbox{on}\quad\Sigma,\\
\ds\lim_{t\rightarrow 0^{+}}{}_{0}I^{1-\alpha}_{t}y(x,t)
=y_{_{0}}(x) \quad &\mbox{in}\quad\Omega,
\end{cases}
\end{equation}
where ${}_{0}D^{\alpha}_{t}$ and ${}_{0}I^{1-\alpha}_{t}$ denote, respectively, 
the Riemann--Liouville fractional order derivative and integral 
with respect to time $t$. For details on these operators, see e.g. 
\cite{IP-1999,AAK-HMS-JJT-2006}. Here we just recall their definition: 
\begin{equation*}
{}_{0}I^{1-\alpha}_{t}y(x,t) = \ds\frac{1}{\Gamma(\alpha)}
\int_{0}^{t}(t-s)^{\alpha-1}y(x,s)ds
\end{equation*}
and
\begin{equation*}
{}_{0}D^{\alpha}_{t}y(x,t) = \ds\frac{\partial}{\partial t} 
{}_{0}I^{1-\alpha}_{t}y(x,t), 
\end{equation*}
where $0<\alpha<1$. The second order operator 
$\mathcal{A}$ in \eqref{sys1} is linear 
with dense domain such that the coefficients do not depend on $t$. 
It generates a $C_{0}$-semi-group $(S(t))_{_{t\geq 0}}$ 
on the Hilbert space $Y:=L^{2}(\Omega)$. We refer the reader 
to Engel and Nagel \cite{KJE-RN-2006} and Renardy and Rogers 
\cite{MR-RCR-2004} for more properties on operator $\mathcal{A}$. 
The initial datum $y_{_{0}}$ is in $Y$. The operator 
$\mathcal{B}: \mathbb{R}^{m}\rightarrow Y$ is the control 
operator, which depends on the number $m$ of actuators 
and $u\in L^{2}(0,T;\mathbb{R}^{m})$.


\section{Preliminaries}
\label{sec:3}

In this section, we recall some notions and facts needed in the sequel.
We begin with the concept of mild solution, that has been used in 
fractional calculus in several different contexts 
\cite{MR3403507,MR3172421,MR3244467}.

\begin{definition}[See, e.g., \cite{MR3316531,GF-CYQ-KC-2015,FM-PP-RG-2007}]
\label{Def1}
For any given function $f\in L^{2}(0,T;Y)$ and $\alpha\in (0,1)$, 
we say that function $g\in L^{2}(0,T;Y)$ is a mild solution of the system
\begin{equation*}
\begin{cases}
{}_{0}D^{\alpha}_{t}g(t) = \mathcal{A} g(t) + f(t), \qquad t\in[0,T],\\
\ds\lim_{t\rightarrow 0^{+}} {}_{0}I^{1-\alpha}_{t} g(t) = g_{_{0}}\in Y,
\end{cases}
\end{equation*}
if it satisfies
\begin{equation*}
g(t) = K_{\alpha}(t)g_{_{0}} 
+ \ds\int_{0}^{t}(t-s)^{\alpha-1}K_{\alpha}(t-s)f(s)ds,
\end{equation*}
where 
$$
K_{\alpha}(t) = \alpha\ds
\int_{0}^{\infty}\theta\phi_{\alpha}(\theta)S(t^{\alpha}\theta)d\theta
$$ 
with
$$
\phi_{\alpha}(\theta) = \ds\frac{1}{\alpha}
\theta^{-1-1/\alpha}\varphi_{\alpha}(\theta^{-1/\alpha})
$$ 
and $\varphi_{\alpha}$ is the probability density function defined by
\begin{equation*}
\varphi_{\alpha}(\theta) = \ds\frac{1}{\pi}
\sum_{n=1}^{\infty}(-1)^{n-1}\theta^{-\alpha n-1}\frac{\Gamma(n\alpha+1)}{n!}
\sin(n\pi\alpha),\quad \theta>0.
\end{equation*}	
\end{definition}

\begin{remark}
The probability density function satisfies
\begin{equation*}
\displaystyle\int_{0}^{\infty}\varphi_{\alpha}(\theta)d\theta = 1.
\end{equation*}  
Moreover,
$$
\int_{0}^{\infty}\theta^{v}\phi_{\alpha}(\theta)d\theta 
= \frac{\Gamma(1+v)}{\Gamma(1+\alpha v)},\quad v\geq 0\cdot
$$
\end{remark}

For results on existence and uniqueness of mild solutions
for a class of fractional neutral evolution equations with nonlocal conditions,
we refer to Zhou and Jiao \cite{YZ-FJ-2010}. Here we note
that a mild solution of system \eqref{sys1} can be written as
\begin{equation}
\label{eq8}
y(x,T;u) = T^{\alpha-1}K_{\alpha}(T)y_{_{0}}\\
+ \ds\int_{0}^{T}(T-s)^{\alpha-s}K_{\alpha}(T-s)\mathcal{B}u(s)ds.
\end{equation}

\begin{remark}
Throughout the paper, 
$y(x,t)$ is the variable of system \eqref{sys1}. 
It depends on time $t$ and space $x$.
We use $y(x,t;u)$ as the solution of system \eqref{sys1} 
when it is excited with a control $u$. 
We also denote it formally by $y(u)$.
\end{remark}

Let us define the operator $H:L^{2}(0,T;\mathbb{R}^{m})\rightarrow Y$ by
\begin{equation}
\label{eq9}
Hu = \ds\int_{0}^{T}(T-s)^{\alpha-1}K_{\alpha}(T-s)\mathcal{B}u(s)ds
\end{equation}
for all $u\in L^{2}(0,T;\mathbb{R}^{m})$
and assume that $(S^{*}(t))_{_{t\geq 0}}$ is a strongly continuous 
semi-group generated by the adjoint operator of $\mathcal{A}$ 
on the state space $Y$. Let $\langle \cdot ,\cdot\rangle$ 
be the duality pairing of space $Y$. It is easy to see that 
$\langle Hu , v\rangle  = \langle u , H^{*}v\rangle$. 
Indeed, for all $v\in Y$ one has
\begin{equation}
\label{eq:ref01}
\begin{array}{lll}
\left\langle Hu , v\right\rangle 
&= \left\langle\displaystyle\int_{0}^{T}(T-s)^{\alpha-1}K_{\alpha}(T-s)\mathcal{B}u(s)ds,v\right\rangle\\
& = \displaystyle\int_{0}^{T} \langle (T-s)^{\alpha-1}K_{\alpha}(T-s)\mathcal{B}u(s),v\rangle ds\\
& = \displaystyle\int_{0}^{T} \left\langle u(s),
\mathcal{B}^{*}(T-s)^{\alpha-1}K_{\alpha}^{*}(T-s)v\right\rangle ds\\
& = \langle u, H^{*}v \rangle.
\end{array}
\end{equation}
For any $v\in Y$, it follows from \eqref{eq:ref01} that
\begin{equation*}
H^{*}v = \mathcal{B}^{*}(T-s)^{\alpha-1}K_{\alpha}^{*}(T-s)v,
\end{equation*}
where $\mathcal{B}^{*}$ is the adjoint operator of $\mathcal{B}$, and
\[
K_{\alpha}^{*} = \alpha\ds\int_{0}^{\infty}
\theta\phi_{\alpha}(\theta)S^{*}(t^{\alpha}\theta)d\theta.
\]
Let $\omega\subset\Omega$ be a given region of positive Lebesgue 
measure. We define the restriction operator $\chi_{_{\omega}}$ 
and its adjoint $\chi^{*}_{_{\omega}}$ by 
\begin{equation*}
\begin{array}{rll}
\hspace{-1cm} \chi_{_{\omega}} : L^{2}(\Omega) & \rightarrow L^{2}(\omega)\\
y & \mapsto \chi_{_{\omega}}y = y\arrowvert_{_{\omega}}
\end{array}
\end{equation*}
and
\begin{equation*}
(\chi_{_{\omega}}^{*}y)(x) = \left\{
\begin{array}{ll}
& y(x) \qquad  x \in \omega \\
& 0 \qquad  x \in \Omega\setminus\omega.
\end{array}
\right.  
\end{equation*}

In order to prove our results, the following lemma is needed.

\begin{lemma}[See\cite{KM-2009}]
\label{lemma1}
Let the reflection operator $\mathcal{Q}$ 
on the interval $[0,T]$ be defined by
\begin{equation*}
\mathcal{Q}h(t):=h(T-t),
\end{equation*}
for some function $h$ which is differentiable 
and integrable in the Riemann--Liouville sense.
Then the following relations hold:
\begin{equation*}
\mathcal{Q}{} \, \, _{0}I^{\alpha}_{t}h(t) 
= {}_{t}I^{\alpha}_{T}\mathcal{Q}h(t), 
\qquad \mathcal{Q}{} \, \, _{0}D^{\alpha}_{t}h(t) 
= {}_{t}D^{\alpha}_{T}\mathcal{Q}h(t)
\end{equation*}
and
\begin{equation*}
{}_{0}I^{\alpha}_{t}\mathcal{Q}h(t) 
= \mathcal{Q}{}  \, \, _{t}I^{\alpha}_{T}h(t), 
\qquad {}_{0}D^{\alpha}_{t}\mathcal{Q}h(t) 
= \mathcal{Q}{} \, \, _{t}D^{\alpha}_{T}h(t).
\end{equation*}
\end{lemma}


\section{Enlarged Controllability and Characterization}
\label{sec:4}

We define exact enlarged controllability (EEC) as follows.

\begin{definition}
\label{Def2}
Given a final time $T>0$ and a suitable functional space $G$, 
we say that system \eqref{sys1} is exact enlarged controllable if, 
for every $y_{_{0}}$ in a suitable functional space, 
there exists a control $u$ such that
\begin{equation}
\label{def-enlarg-contr}
\chi_{_{\omega}}y(\cdot,T;u)\in G.
\end{equation}
\end{definition}

\begin{remark}
The concept of exact enlarged controllability depends, obviously, on $G$.
If $G = \{0\}$, then one gets from Definition~\ref{Def2}
the classical notion of controllability.
\end{remark}

\begin{remark}
If there is exact enlarged controllability depending on $G$, 
then there exists an infinity number of controls $u$ 
that verify \eqref{def-enlarg-contr}. 
However, exact enlarged controllability relatively 
to $G$ does not implies exact controllability. 
\end{remark}

\begin{example}
Let $\omega$ be an open subset of $\Omega$ 
and let us search $u$ such that
${}_{0}D^{\alpha}_{t}y(x,T;u) = 0$ in $\omega$.
In this case, the space $G$ can be given as
$G = G_{0}\times\{\mbox{the whole domain}\}$,
where $G_{0}$ is the space of null functions in $\omega$. 
\end{example}

In our case, we take $G$ as a sub-vectorial closed space 
of $Y:=L^{2}(\Omega)$. The following result holds.

\begin{theorem}
Our system is exact enlarged controllable (i.e., system \eqref{sys1} is exactly 
$G$-controllable in $\omega$ in the sense of Definition~\ref{Def2}) if and only if 
\begin{equation}
\label{eq18}
G-\left\lbrace \chi_{_{\omega}}T^{\alpha-1}
K_{\alpha}(T)y_{_{0}}\right\rbrace 
\cap\Ima\chi_{_{\omega}}H\neq\emptyset.
\end{equation}
\end{theorem}

\begin{proof}
Suppose \eqref{eq18} holds.
Then, there exists 
$$
z\in G-\left\lbrace
\chi_{_{\omega}}T^{\alpha-1}K_{\alpha}(T)y_{_{0}}\right\rbrace
$$ 
such that 
$z\in\Ima\chi_{_{\omega}}H$. 
So, there exists $u\in L^{2}(0,T;\mathbb{R}^{m})$
such that $z = \chi_{_{\omega}}Hu$. Hence, 
$$
z = \chi_{_{\omega}}Hu\in G-\left\lbrace
\chi_{_{\omega}}T^{\alpha-1}K_{\alpha}(T)
y_{_{0}}\right\rbrace.
$$
Therefore, $y(u)\in G$ and we have EEC in $\omega$. Conversely, assume that one has 
EEC of \eqref{sys1} in $\omega$, which means that $\chi_{_{\omega}}y(u)\in G$. 
Using \eqref{eq8} and \eqref{eq9}, we have
$$
\chi_{_{\omega}}y(u) = \chi_{_{\omega}}T^{\alpha-1}
K_{\alpha}(T)y_{_{0}} + \chi_{_{\omega}}Hu.
$$
Let us denote $w = \chi_{_{\omega}}y(u) - \chi_{_{\omega}} 
T^{\alpha-1}K_{\alpha}(T)y_{_{0}} = \chi_{_{\omega}}Hu$. 
Then, one has $w\in\Ima\chi_{_{\omega}}H$ and
$$
w\in G-\left\lbrace \chi_{_{\omega}}T^{\alpha-1}
K_{\alpha}(T)y_{_{0}}\right\rbrace.
$$ 
We just proved \eqref{eq18}.
\end{proof}


\section{Fractional HUM Approach}
\label{sec:5}

We now extend the Hilbert uniqueness method (HUM), 
introduced by Lions in \cite{JLL-1988}, to the fractional setting \eqref{sys1} 
and try to compute the (optimal) control that steers system \eqref{sys1}
into $G$. First of all, we prove under what conditions we can find 
enlarged controllability. Then we compute the (optimal) control 
that steers our system into $G$. The reader interested in the
HUM approach, in the context of fractional calculus, is referred to
\cite{GF-CYQ-KC-2016-2,MR2139969,MR1313522,MR2358729}.


\subsection{The Case of a Zone Actuator}
\label{sub-sect1}

Let us consider system \eqref{sys1} excited with a zone actuator $(D,f)$, 
where $D\subseteq\Omega$ is the support of the actuator and $f$ 
its spatial distribution. For details about actuators, we refer 
to \cite{AE-AJP-1988,GF-CYQ-KC-2016}. System \eqref{sys1} 
can be written as follows:
\begin{eqnarray}
\label{sys19}
\begin{cases}
\ds {}_{0}D^{\alpha}_{t}z(x,t) = \mathcal{A}z(x,t)  
+ \chi_{_{D}}f(x)u(t) \quad &\mbox{in}\quad Q \\
z(\xi,t)=0  \quad &\mbox{on}\quad\Sigma \\
\ds\lim_{t\rightarrow 0^{+}}{}_{0}I^{1-\alpha}_{t}z(x,t)
=z_{_{0}}(x) \quad &\mbox{in}\quad\Omega.
\end{cases}
\end{eqnarray}

The main question we answer is the following one.
Does there exist a (minimum norm) control $u\in L^{2}(0,T;\mathbb{R}^{m})$ 
such that $\chi_{_{\omega}}y(x,T;u)\in G$?

Let us introduce $G^{\circ}$ as the polar space of $G$. Then,
$\varphi_{_{0}}\in G^{\circ}$, that is, 
$\langle\varphi_{_{0}},\phi\rangle = 0$ for all $\phi\in G$,
where $\langle \cdot , \cdot \rangle$ denotes the scalar product in $Y$.
For $\varphi_{_{0}}\in G^{\circ}$, we consider the following backward system: 
\begin{eqnarray}
\label{sys20}
\begin{cases}
\mathcal{Q} \, \, \ds {}_{t}D^{\alpha}_{T}\varphi(x,t) 
= \mathcal{A}^{*}\mathcal{Q}\varphi(x,t)   \quad &\mbox{in}\quad Q \\
\varphi(\xi,t) = 0 \quad &\mbox{on}\quad \Sigma\\
\ds\lim_{t\rightarrow 0^{+}}\mathcal{Q}{} \,\, _{t}I^{1-\alpha}_{T}\varphi(x,t)
=\chi_{_{\omega}}^{*}\varphi_{_{0}}(x) \quad &\mbox{in}\quad\Omega\;,
\end{cases}
\end{eqnarray}
where $\mathcal{Q}$ is the reflection operator on the interval $[0,T]$ 
defined in Lemma~\ref{lemma1}. Hence, system \eqref{sys20} 
can be rewritten as
\begin{eqnarray}
\label{sys21}
\begin{cases}
{}_{0}D^{\alpha}_{t}\mathcal{Q} \varphi(x,t) 
= \mathcal{A}^{*}\mathcal{Q}\varphi(x,t)   
&\mbox{in}\  Q \\
\varphi(\xi,t) = 0 &\mbox{on}\ \Sigma\\
\ds\lim_{t\rightarrow 0^{+}}{}_{0}I^{1-\alpha}_{t}\mathcal{Q}
\left[ (T-t)^{1-\alpha}\varphi(x,t)\right] 
= \chi_{_{\omega}}^{*}\varphi_{_{0}}(x) &\mbox{in}\ \Omega.
\end{cases}
\end{eqnarray}
System \eqref{sys21} has a unique mild solution given by
$$
\varphi(x,t) = (T-t)^{\alpha-1}K_{\alpha}^{*}(T-t)
\chi_{_{\omega}}^{*}\varphi_{_{0}}(x).
$$
Moreover, we define the following semi-norm on $G$:
\begin{equation}
\label{eq22}
\varphi_{_{0}}\in G^{\circ} \rightarrow 
\|\varphi_{_{0}}\|_{_{G^{\circ}}}^{^{2}} 
= \ds\int_{0}^{T}\langle f,\varphi(\cdot,t)\rangle_{_{L^{2}(D)}}^{^{2}}dt.
\end{equation}

\begin{lemma}
\label{thm-norm}
Equation \eqref{eq22} defines a norm.
\end{lemma}

\begin{proof}
To prove that \eqref{eq22} is a norm, we show that
$\|\varphi_{_{0}}\|_{_{G^{\circ}}}^{^{2}} = 0
\Leftrightarrow \varphi \equiv 0$.
Writing $\|\varphi_{_{0}}\|_{_{G^{\circ}}}^{^{2}} = 0$, 
is equivalent to $\langle f,\varphi(\cdot,t)\rangle = 0$. 
It follows from the uniqueness theorem 
in \cite{EB-2012} that $\varphi\equiv 0$.	
\end{proof}

We also consider the problem:
\begin{eqnarray*}
\begin{cases}
{}_{0}D^{\alpha}_{t}\psi(x,t) 
= \mathcal{A}\psi(x,t)  + \chi_{_{D}}f(x)u(t) 
\quad &\mbox{in}\quad Q \\
\psi(\xi,t)=0  \quad &\mbox{on}\quad\Sigma \\
\ds\lim_{t\rightarrow 0^{+}}{}_{0}I^{1-\alpha}_{t}\psi(x,t)
=z_{_{0}}(x) \quad &\mbox{in}\quad\Omega.
\end{cases}
\end{eqnarray*}
We obtain $\psi$ such that
$\psi : [0,T] \rightarrow Y$
is continuous. If we find $\varphi_{_{0}}$ such that 
$\varphi_{_{0}}\in G^{\circ}$ and 
\begin{equation}
\label{eq24}
\chi_{_{\omega}}\psi(T)\in G,
\end{equation}
then
\begin{equation}
\label{eq25}
u = \langle f,\varphi(\cdot,t)\rangle_{_{L^{2}(D)}}.
\end{equation}
The control \eqref{eq25} ensures EEC and we have
$\psi = y(u)$, where $y(u)$ is a formal notation of the solution $y(x,T;u)$.
In order to easily explain \eqref{eq24} with an equation, we define
\begin{equation}
\label{eq27}
M\varphi_{_{0}} 
= \chi_{_{\omega}}\psi(T),
\end{equation}
where $M$ is an affine operator from $G^{\circ}$ 
to the orthogonal $G^{\bot}$ of $G$.
Thus, all return to solve the equation
$M\varphi_{_{0}} = 0$.
Let us decompose $M$ into a linear part and a constant one:
\begin{equation*}
M\varphi_{_{0}} = \chi_{_{\omega}}\left( \psi_{_{0}}(T) 
+ \psi_{_{1}}(T) \right), 
\end{equation*} 
where
\begin{eqnarray*}
\begin{cases}
{}_{0}D^{\alpha}_{t}\psi_{_{0}}(x,t) 
= \mathcal{A}\psi_{_{0}}(x,t)  + \chi_{_{D}}f(x)\langle 
f,\varphi(\cdot,t)\rangle_{_{L^{2}(D)}} & \mbox{in}\ Q \\
\psi_{_{0}}(\xi,t)=0  &\mbox{on}\ \Sigma \\
\ds\lim_{t\rightarrow 0^{+}}{}_{0}I^{1-\alpha}_{t}\psi_{_{0}}(x,t)
=z_{_{0}}(x) &\mbox{in}\ \Omega
\end{cases}
\end{eqnarray*}
and
\begin{eqnarray*}
\begin{cases}
{}_{0}D^{\alpha}_{t}\psi_{_{1}}(x,t) 
= \mathcal{A}\psi_{_{1}}(x,t)  \quad &\mbox{in}\quad Q \\
\psi_{_{1}}(\xi,t)=0  \quad &\mbox{on}\quad\Sigma \\
\ds\lim_{t\rightarrow 0^{+}}{}_{0}I^{1-\alpha}_{t}
\psi_{_{1}}(x,t)=z_{_{0}}(x) \quad &\mbox{in}\quad\Omega.
\end{cases}
\end{eqnarray*}
Then,
\begin{equation*}
M\varphi_{_{0}} = \chi_{_{\omega}}\psi_{_{0}}(T) 
+ \chi_{_{\omega}}\psi_{_{1}}(T).
\end{equation*}
Let $\Lambda$ be the operator given by
\begin{equation}
\label{eq33}
\begin{array}{rll}
\Lambda : G^{\circ} & \rightarrow  G^{\bot}\\
\varphi_{_{0}} & \mapsto \chi_{_{\omega}}\psi_{_{0}}(T).
\end{array}
\end{equation}
With these notations, the exact enlarged controllability 
problem leads to the solution of equation
\begin{equation}
\label{eq34}
\Lambda\varphi_{_{0}} = -\chi_{_{\omega}}\psi_{_{1}}(T).
\end{equation}
The following result holds.

\begin{theorem}
\label{thm1:eecG}
System \eqref{sys19} is exact enlarged controllable relatively to $G$.
Moreover, the control 
\begin{equation*}
u^{*}(t) = \langle f,\varphi(\cdot,t)\rangle_{_{L^{2}(D)}}
\end{equation*}
steers the system into $G$.
\end{theorem}

\begin{proof}
From Lemma~\ref{thm-norm}, we have that \eqref{eq22} is a norm. 
Now, we show that \eqref{eq34} admits a unique solution in $G$. 
For any $\varphi_{_{0}}\in G^{\circ}$, by \eqref{eq33} it follows that
\[
\begin{array}{ll}
\langle\varphi_{_{0}},\Lambda\varphi_{_{0}}\rangle
&= \langle\varphi_{_{0}},\chi_{_{\omega}}\psi_{_{1}}(T)\rangle\\
&=  \left\langle \varphi_{_{0}},\chi_{_{\omega}}\ds
\int_{0}^{T}(T-s)^{\alpha-1}K_{\alpha}(T-s)\chi_{_{D}}f(\cdot)\langle 
f,\varphi(\cdot,s)\rangle_{_{L^{2}(D)}}ds\right\rangle \\
&= \ds\int_{0}^{T}\left\langle (T-s)^{\alpha-1}
K_{\alpha}^{*}(T-s)\chi_{_{\omega}}^{*}\varphi_{_{0}},\chi_{_{D}}f(\cdot)\langle f,\varphi(\cdot,s)\rangle_{_{L^{2}(D)}}\right\rangle ds\\
&= \ds\int_{0}^{T} \left\langle \chi_{_{\omega}}^{*}f(.)\varphi(\cdot,s), 
\langle f,\varphi(\cdot,s)\rangle_{_{L^{2}(D)}} \right\rangle ds\\
&= \ds\int_{0}^{T}\left\langle \langle 
f,\varphi(\cdot,s)\rangle_{_{L^{2}(D)}}, \langle f, 
\varphi(\cdot,s)\rangle_{_{L^{2}(D)}}\right\rangle ds\\
&= \ds\int_{0}^{T}\|\langle f,\varphi(\cdot,s)\rangle_{_{L^{2}(D)}}\|^{^{2}}ds\\
&= \|\varphi_{_{0}}\|_{_{G^{\circ}}}^{^{2}}.
\end{array}
\]
Existence of a unique solution follows from 
\cite[Theorem~1.1]{book:Lions:1971}.
\end{proof}


\subsection{The Case of a Pointwise Actuator}

Consider now system \eqref{sys1} with a pointwise internal actuator, 
which can be written in the form
\begin{eqnarray}
\label{sys36}
\begin{cases}
\ds {}_{0}D^{\alpha}_{t}y(x,t) 
= \mathcal{A}y(x,t)  + \delta(x-b)u(t) 
\quad &\mbox{in}\quad Q \\
y(\xi,t)=0  
\quad &\mbox{on}\quad\Sigma \\
\ds\lim_{t\rightarrow 0^{+}}{}_{0}I^{1-\alpha}_{t}y(x,t)
=y_{_{0}}(x) \quad &\mbox{in}\quad\Omega\;,
\end{cases}
\end{eqnarray}
where $b$ is the position of the actuator. 
For any $\varphi_{_{0}}\in G^{\circ}$, we consider 
system \eqref{sys20} and we define the following semi-norm:
\begin{equation*}
\varphi_{_{0}}\in G^{\circ} 
\rightarrow \|\varphi_{_{0}}\|_{_{G^{\circ}}}^{^{2}} 
= \ds\int_{0}^{T}\langle\varphi(b,t)\rangle_{_{L^{2}(D)}}^{^{2}}dt.
\end{equation*}
Using similar arguments as in Section~\ref{sub-sect1},
we can easily prove that this semi-norm is indeed a norm. 
Let us consider $u(t) = \varphi(b,t)$ and the following two systems:
\begin{eqnarray*}
\begin{cases}
{}_{0}D^{\alpha}_{t}\psi_{_{0}}(x,t) 
= \mathcal{A}\psi_{_{0}}(x,t) + \delta(x-b)\varphi(b,t)\quad & \mbox{in}\quad Q \\
\psi_{_{0}}(\xi,t)=0  \quad &\mbox{on}\quad\Sigma \\
\ds\lim_{t\rightarrow 0^{+}}{}_{0}I^{1-\alpha}_{t}\psi_{_{0}}(x,t)=0 
\quad &\mbox{in}\quad\Omega
\end{cases}
\end{eqnarray*}
and
\begin{eqnarray*}
\begin{cases}
{}_{0}D^{\alpha}_{t}\psi_{_{1}}(x,t) 
= \mathcal{A}\psi_{_{1}}(x,t)  \quad &\mbox{in}\quad Q\\
\psi_{_{1}}(\xi,t)=0  \quad &\mbox{on}\quad\Sigma\\
\ds\lim_{t\rightarrow 0^{+}}{}_{0}I^{1-\alpha}_{t}
\psi_{_{1}}(x,t)=y_{_{0}}(x) \quad &\mbox{in}\quad\Omega.
\end{cases}
\end{eqnarray*}
Then, the exact enlarged controllability problem 
is equivalent to solve equation
\begin{equation*}
\Lambda\varphi_{_{0}} = -\chi_{_{\omega}}\psi_{_{1}}(T).
\end{equation*}
Similarly to the proof of Theorem~\ref{thm1:eecG}, 
for the case of a zone actuator, 
we prove exact enlarged controllability
for a pointwise actuator.

\begin{theorem}
System \eqref{sys36} is exact enlarged controllable relatively to $G$. 
Moreover, the control 
\begin{equation*}
u^{*}(t) = \varphi(b,t)
\end{equation*}
steers the system into $G$.
\end{theorem}


\section{Minimum Energy Control}
\label{sec:6}

The study of fractional optimal control problems
is a subject under strong development: 
see \cite{MR3673702,MR3673710,MR3673711,MyID:387} 
and references therein. In this section, 
inspired by the results of \cite{MR2721980,MyID:181}, 
we show that the steering controls found 
in Section~\ref{sec:5} are minimizers of a suitable
optimal control problem. For that, let us consider 
the following minimization problem: 
\begin{eqnarray}
\label{PB42}
\begin{cases}
\inf \mathcal{J}(u) = \ds\frac{1}{2}
\int_{0}^{T}\|u\|_{_{U}}^{^{2}}dt\\
u\in U_{ad},	
\end{cases}
\end{eqnarray}
where 
$$
U_{ad} = \left\lbrace u\in U=L^{2}(0,T;\mathbb{R}^{m})
\, | \, \chi_{_{\omega}}y(u)\in G \right\rbrace. 
$$
The proof of our Theorem~\ref{thm:PT} is based on 
a penalization method \cite{MR3259239,MR2824729}.

\begin{theorem}
\label{thm:PT}
Assume that $U_{ad}$ is nonempty with an exact enlarged controllable
system relatively to $G$ given by \eqref{sys19} or \eqref{sys36}. 
Then the optimal control problem \eqref{PB42} 
has a unique solution given by
$u^{*}(t) = \langle f,\varphi(t)\rangle$,
in case of a zone actuator, and by
$u^{*}(t) = \varphi(b,t)$, in case of a pointwise actuator.
Such control ensures the transfer of system \eqref{sys1} 
into $G$ with a minimum energy cost in the sense of $\mathcal{J}$. 
\end{theorem}

\begin{proof}
Let $\epsilon > 0$. Suppose that there is EEC relatively to $G$ 
and consider the following problem:
\begin{eqnarray}
\label{sys43}
\begin{cases}
{}_{0}D^{\alpha}_{t}z(x,t) - \mathcal{A}z(x,t) 
- \chi_{_{D}}f(x)u(t)\:&\in L^{2}(Q) \\
z(\xi,t)=0  \quad & \mbox{on}\quad\Sigma \\
\ds\lim_{t\rightarrow 0^{+}}{}_{0}I^{1-\alpha}_{t}z(x,t)=0 
\quad &\mbox{in}\quad\Omega\\
z(T,u)\in G.
\end{cases}
\end{eqnarray}
The set $S$ of pairs $(z,u)$ verifying \eqref{sys43} is nonempty. 
Consider the penalized problem of \eqref{PB42} given by
\begin{equation}
\label{PB44}
\begin{cases}
\inf \mathcal{J}_{\epsilon}(u,z)\\
(u,z)\in S,
\end{cases}
\end{equation}
where
\begin{equation*}
\mathcal{J}_{\epsilon} = \ds\frac{1}{2}\int_{0}^{T}u^{2}(t)dt\\
+ \frac{1}{2\epsilon}\int_{Q}\left( {}_{0}D^{\alpha}_{t}z(x,t)
-\mathcal{A}z(x,t)-\chi_{_{D}}f(x)u(t)\right) ^{2}dQ.
\end{equation*}
Let $\{u_{\epsilon},z_{\epsilon}\}$ be the solution of \eqref{PB44} 
and let us define
$$
p_{\epsilon} = -\ds\frac{1}{\epsilon}\left( 
{}_{0}D^{\alpha}_{t}z_{\epsilon}(x,t)
-\mathcal{A}z_{\epsilon}(x,t)
-\chi_{_{D}}f(x)u_{\epsilon}(t)\right).
$$
Because we did assume that $U_{ad}$ is nonempty, we have
\begin{equation*}
0<\mathcal{J}_{\epsilon}(u_{\epsilon},z_{\epsilon}) 
= \inf\mathcal{J}_{\epsilon}(u,z) < \inf\mathcal{J}_{\epsilon}(u)<\infty
\quad \mbox{ for }\: u\in U_{ad}, 
\end{equation*}
where $\mathcal{J}_{\epsilon}(u) = \ds\frac{1}{2}\int_{0}^{T} u^{2}(t)dt$.
Therefore, 
\begin{eqnarray}
\label{cond-1}
\begin{cases}
\|u_{\epsilon}\|\leq C,\\
\|{}_{0}D^{\alpha}_{t}z(x,t) - \mathcal{A}z(x,t) 
- \chi_{_{D}}f(x)u(t)\|\leq C\sqrt{\epsilon},
\end{cases}
\end{eqnarray}
where $C$ represents various positive constants independent 
of $\epsilon$. It follows from \eqref{cond-1} that
\begin{equation*}
\|{}_{0}D^{\alpha}_{t}z(x,t) - \mathcal{A}z(x,t)\|
\leq C(1+\sqrt{\epsilon}).
\end{equation*} 
Hence, when $\epsilon\rightarrow 0$, we have that $u_{\epsilon}$ 
is bounded and we can extract a sequence such that
$$
\begin{array}{cl}
u_{\epsilon}\rightharpoonup\tilde{u}
\quad &\mbox{weakly in} \quad U\\
z_{\epsilon}\rightharpoonup z
\quad &\mbox{weakly in} \quad L^{2}(Q).
\end{array}
$$
Using the semi-continuity of $\mathcal{J}$, one has
$$
\mathcal{J}(u^{*}) \leq \lim\inf\mathcal{J}_{\epsilon}(u_{\epsilon}) 
\leq \lim\inf\mathcal{J}_{\epsilon}(u_{\epsilon},z_{\epsilon}).
$$
Then,
\begin{equation*}
\mathcal{J}(u^{*}) 
= \inf\mathcal{J}(u), 
\quad u \in U_{ad},
\end{equation*}
and $u^{*} = \tilde{u}$. With respect 
to problem \eqref{PB44}, we have  
\begin{equation*}
\int_{0}^{T}u_{\epsilon}(t)u(t)dt + \int_{Q}\left\langle
p_{\epsilon},{}_{0}D_{t}^{\alpha}\eta(x,t)-\mathcal{A}\eta(x,t)\right\rangle dQ\\ 
= -\int_{Q}\langle p_{\epsilon},f(x)\rangle u(t)dQ.
\end{equation*}
For $u\in U_{ad}$ and $\eta$ such that
\begin{eqnarray*}
\begin{cases}
{}_{0}D^{\alpha}_{t}\eta(x,t) - \mathcal{A}\eta(x,t)  
= \chi_{_{D}}f(x)u(t) \quad &\mbox{in}\quad Q \\
\eta(\xi,t)=0  \quad &\mbox{on}\quad\Sigma \\
\ds\lim_{t\rightarrow 0^{+}}{}_{0}I^{1-\alpha}_{t}\eta(x,t)=0 
\quad &\mbox{in}\quad\Omega\\
\eta(T)\in G\,,
\end{cases}
\end{eqnarray*}
we deduce that $p_{\epsilon}$ verifies
\begin{eqnarray*}
\begin{cases}
{}_{0}D^{\alpha}_{t}p_{\epsilon}(x,t) - \mathcal{A}p_{\epsilon}(x,t)  
= \chi_{_{D}}f(x)\langle p_{\epsilon},f\rangle_{_{L^{2}(D)}} 
\quad &\mbox{in}\quad Q \\
p_{\epsilon}(\xi,t)=0  \quad &\mbox{on}\quad\Sigma \\
\ds\lim_{t\rightarrow 0^{+}}{}_{0}I^{1-\alpha}_{t}p_{\epsilon}(x,t)=0 
\quad &\mbox{in}\quad\Omega
\end{cases}
\end{eqnarray*}
with $\langle\eta(T),p_{\epsilon}(T)\rangle = 0$ for all $\eta$ 
such that $\eta(T)\in G$. Then, $p_{\epsilon}(T)\in G^{\circ}$. 
If we suppose that 
$$
\int_{0}^{T}\langle p_{\epsilon},f\rangle^{^{2}} dt 
\geq c\|p_{\epsilon}(T)\|_{H_{0}^{1}(\Omega)}^{^{2}},
$$
then we can switch to the limit when $\epsilon\rightarrow 0$. 
Moreover, if we have exact enlarged controllability relatively 
to $G$, then 
\begin{eqnarray*}
\begin{cases}
{}_{0}D^{\alpha}_{t}z(x,t) - \mathcal{A}z(x,t)  
= \chi_{_{D}}f(x)u(t) \quad &\mbox{in}\quad Q \\
z(\xi,t)=0  \quad &\mbox{on}\quad\Sigma \\
\ds\lim_{t\rightarrow 0^{+}}{}_{0}I^{1-\alpha}_{t}z(x,t)
=z_{_{0}}(x) \quad &\mbox{in}\quad\Omega\\
{}_{0}D^{\alpha}_{t}p(x,t) - \mathcal{A}p(x,t) 
= \chi_{_{D}}f(x)\langle p,f\rangle_{_{L^{2}(D)}} 
\quad &\mbox{in}\quad Q \\
p(\xi,t)=0  \quad &\mbox{on}\quad\Sigma\\
p(T)\in G^{\circ}.
\end{cases}
\end{eqnarray*}
Thus, we take $p(T)\in G^{\circ}$ and we introduce the  
solution $\varphi$ of \eqref{sys21}. Then, $\psi = z$ 
if $\psi(T)\in G$, which proves that \eqref{eq27} 
has a unique solution for $\varphi_{_{0}}\in G^{\circ}$.
\end{proof}


\section{Conclusions}
\label{sec:7}

This paper deals with the notion of regional 
exact enlarged controllability for Riemann--Liouville 
time fractional diffusion systems.
Our results extend the ones in
\cite{GF-CYQ-KC-2016-2,GF-CYQ-KC-IP-2016,GF-CYQ-KC-2016-3}. 
They can be extended to complex fractional-order 
distributed parameter dynamic systems. Other difficult 
questions are still open and deserving further investigations,
e.g., the problem of boundary enlarged controllability 
for fractional systems; and the problem of gradient enlarged 
controllability/observability for fractional order 
distributed parameter systems. These and other questions,
as to give numerical results and a real application 
to support our theoretical analysis, are being considered 
and will be addressed elsewhere.


\begin{acknowledgment}
This research is part of first author's Ph.D., which is carried out 
at Moulay Ismail University, Meknes. It was essentially finished during 
a one-month visit of Karite to the Department of Mathematics 
of University of Aveiro, Portugal, June 2017. The hospitality of the 
host institution and the financial support of Moulay Ismail University,
Morocco, and CIDMA, Portugal, are here gratefully acknowledged. 
Boutoulout was supported by Hassan II Academy of Science and Technology;
Torres by Portuguese funds through CIDMA and FCT, within project UID/MAT/04106/2013. 
The authors are very grateful to two anonymous referees, 
for their suggestions and invaluable comments.
\end{acknowledgment}



\end{document}